\documentclass[12pt]{amsart}
\usepackage{amssymb,amsfonts,amsmath,graphicx,lineno}
\usepackage{mathrsfs}
\usepackage{anysize}
\usepackage[colorlinks, citecolor=blue, urlcolor=blue]{hyperref}
\usepackage{url}
\usepackage[leqno]{amsmath}
\usepackage{enumitem}
\theoremstyle{plain}
\newtheorem{theorem}{Theorem}[section]
\newtheorem{lemma}[theorem]{Lemma}
\newtheorem{corollary}[theorem]{Corollary}
\newtheorem{proposition}[theorem]{Proposition}
\theoremstyle{definition}

\theoremstyle{remark}
\newtheorem{remark}[theorem]{Remark}
\newtheorem{example}[theorem]{Example}

\newcommand{\N}{\mathbb{N}}
\newcommand{\Z}{\mathbb{Z}}

\title[Topological transitivity and mixing composition of operators]{Topological transitivity and mixing of composition operators}

\medskip
\marginsize{2.5cm}{2.5cm}{1cm}{2cm}
\setlist[itemize]{leftmargin=1cm}

\DeclareRobustCommand{\rchi}{{\mathpalette\irchi\relax}}
\newcommand{\irchi}[2]{\raisebox{\depth}{$#1\chi$}} % inner command, used by \rchi

\begin{document}

\maketitle

\centerline{\scshape Fr\'ed\'eric Bayart}

{\footnotesize
	\centerline{Laboratoire de Math\'ematiques Blaise Pascal}
	\centerline {Universit\'e Clermont Auvergne}
	\centerline{F-63000 Clermont-Ferrand, France.}
	
	\centerline{Frederic.Bayart@uca.fr} }

\smallskip

\centerline{\scshape Udayan B. Darji}

{\footnotesize
	\centerline{Department of Mathematics, University of Louisville}
	\centerline {Louisville, KY 40292, USA}
	\centerline{and}
	\centerline{Ashoka University, Rajiv Gandhi Education City}
	\centerline {Kundli, Rai 131029, India}
	\centerline{ubdarj01@gmail.com} }

\smallskip

\centerline{\scshape Benito Pires \footnote{Partially supported by  S\~ao Paulo Research Foundation (FAPESP) grant \# 2015/20731-5.}}

{\footnotesize
	\centerline{Departamento de Computa\c c\~ao e Matem\'atica, Faculdade de Filosofia, Ci\^encias e Letras}
	\centerline {Universidade de S\~ao Paulo, 14040-901, Ribeir\~ao Preto - SP, Brazil}
	\centerline{benito@usp.br} } 

\marginsize{2.5cm}{2.5cm}{1cm}{2cm}

  \begin{abstract}  Let $X=(X,\mathcal{B},\mu)$ be a $\sigma$-finite measure space and \mbox{$f:X\to X$}  be a measurable transformation 
  such that the composition operator $T_f:\varphi\mapsto \varphi\circ f$ is a bounded  linear operator acting on $L^p(X,\mathcal{B},\mu)$, $1\le p<\infty$.
  We provide a necessary and sufficient condition on $f$ for $T_f$ to be topologically transitive or topologically mixing.
  We also characterize the topological dynamics of composition operators induced by weighted shifts, non-singular odometers and inner functions.
  The results provided in this article hold for composition operators acting on more general Banach spaces of functions. 
    \end{abstract}
    
\section{Introduction}
It is widely known that bounded linear operators acting on  infinite dimensional Banach spaces  may present chaotic behaviour (see \cite{BM,grosse}). Special attention has been devoted to the  study of chaos in the sense of Devaney and Li-Yorke (see \cite{bernardes2, bernardes1, grivaux}). One of the key ingredients of chaos is the notion of topological transitivity. More specifically, a bounded linear operator $T$ acting on a Banach space $\mathcal X$ is {\it topologically transitive} if for any pair $U,V\subset \mathcal X$ of nonempty open sets, there exists an integer $k\ge 0$ such that $T^k(U)\cap V\neq\emptyset$. When $\mathcal X$ is a separable Banach space, topological transitivity is equivalent to being $hypercyclic$, that is, to the existence of a dense  $T$-orbit $\textrm{Orb}\,(\varphi,T)=\{\varphi,T\varphi,T^2\varphi,\ldots\}.$

Here we are interested in the dynamics of the composition operator $T_f:\varphi\mapsto \varphi\circ f$ 
 acting on a Banach space of functions $\mathcal X$. This is a subject that already appeared in the literature for regular functions, like
 holomorphic functions (see for instance \cite{JHS}) or smooth functions (see \cite{AP}). Here, we focus on measurable functions and $L^p$-spaces. More precisely,
 let $(X,\mathcal B,\mu)$ be a $\sigma$-finite measure space and $\left(\mathcal{X},\left\Vert\cdot\right\Vert\right)\subset L^0(\mu)$ be a Banach space of functions defined on $X$. We will always assume that $\mu(X)>0$ and that $\mathcal X$ is a lattice: if $\psi_1,\psi_2$ are measurable functions with $|\psi_1|\leq |\psi_2|$ and 
$\psi_2\in \mathcal{X}$, then $\psi_1\in \mathcal{X}$ and $\left\Vert \psi_1\right\Vert\le \left\Vert \psi_2\right\Vert$. We will say that $\mathcal{X}$ is \emph{admissible}
provided it satisfies the following assumptions
\begin{enumerate}
\item[(H1)] For any $A\in\mathcal B$ with finite measure, the function $\rchi_A$ belongs to $\mathcal{X}$;
\item[(H2)] The set of simple functions which vanish outside a set of finite measure is dense in $\mathcal{X}$;
\item[(H3)] For all $\epsilon>0$, there exists $\delta>0$ such that, for all $\psi\in \mathcal X$, for all $S\in\mathcal B$, $|\psi|\geq 1$ on $S$ and $\|\psi\|\leq\delta$ imply $\mu(S)<{\epsilon}/{2}$;
\item[(H4)] For all $\eta>0$ and for all $M>0$, there exists $\epsilon>0$ such that for all $\psi\in\mathcal X$,
 for all $S\in \mathcal B$, $\mu(S)<\epsilon$, $\psi=0$ on $X\backslash S$ and $\left\vert \psi\right\vert\le 2M$ imply $\|\psi\|<{\eta}/{2}$.
\end{enumerate}  
It is clear that $L^p$-spaces are admissible (in particular, (H3) follows from Markov's inequality). Throughout the paper, we will assume that $f:X\to X$ is nonsingular (namely $\mu\big(f^{-1}(S)\big)=0$
whenever $\mu(S)=0$). This ensures that $\varphi\circ f$ is well-defined for every $\varphi\in L^0(\mu)$. A necessary and suffficient condition for boundedness on $L^p(\mu)$ is the existence of some $c>0$
such that $\mu\big(f^{-1}(B)\big)\leq c\mu(B)$ for all $B\in\mathcal B$ (see \cite{RM}).

Our first result is a necessary and sufficient condition for the composition operator $T_f$ to be topologically transitive. We do not make any extra assumption: neither
 $X$ need to have finite measure nor $f$ has to bimeasurable or injective. 

\begin{theorem}\label{thm1}
Let $(X,\mathcal B,\mu)$ be a $\sigma$-finite measure space and $\mathcal X$ be an admissible Banach space of functions defined on $X$.
Let $f:X\to X$ be measurable such that the composition operator $T_f:\varphi\mapsto\varphi\circ f$ is bounded on $\mathcal X$.
 Then the following assumptions are equivalent:
\begin{enumerate}
\item [(A1)]  $T_f$ is topologically transitive;
\item [(A2)] $f^{-1}(\mathcal B)=\mathcal B$ and, for all $\epsilon>0$, for all $A\in\mathcal B$ with finite measure, there exist $B\subset A$ measurable, $k\geq 1$ and $C\in\mathcal B$ such that 
\begin{equation*}
\mu(A\backslash B)<\epsilon,\ \mu(f^{-k}(B))<\epsilon,\ f^k(B)\subset C\textrm{ and }\mu(C)<\epsilon.
\end{equation*}
\end{enumerate}
\end{theorem}

We mention that a version of Theorem \ref{thm1} has been already given by Kalmes in \cite[Theorem 2.4]{Kal07} under additional assumptions. More specifically, he assumes that $X$ is $\sigma$-compact,
 $\mu$ is locally finite and $f:X\to X$ is injective and continuous. As we show throughout this article, there are plenty of topologically transitive composition operators
 whose underlying map $f:X\to X$ is not continuous or is defined on a space $X$ that it is not $\sigma$-compact.
  
\smallskip
  
 Another key ingredient of chaos is the notion of topological mixing. A bounded linear operator $T:\mathcal X\to \mathcal X$ of a Banach space $\mathcal X$ is {\it topologically mixing}
 if for any pair $U,V\subset \mathcal X$ of nonempty open sets, there exists an integer $k_0\ge 0$ such that $T^k(U)\cap V\neq\emptyset$ for every $k\ge k_0$. 
 Our second result is the following.
 
 \begin{theorem}\label{thm2}
Let $(X,\mathcal B,\mu)$ be a $\sigma$-finite measure space and $\mathcal X$ be an admissible Banach space of functions defined on $X$.
Let $f:X\to X$ be measurable such that the composition operator $T_f:\varphi\mapsto\varphi\circ f$ is bounded on $\mathcal X$.
 Then the following assumptions are equivalent:
\begin{enumerate}
\item [(B1)] $T_f$ is topologically mixing;
\item [(B2)] $f^{-1}(\mathcal B)=\mathcal B$ and, for all $\epsilon>0$,  for all $A\in\mathcal B$ with finite measure, there exist $k_0\geq 1$ and two sequences $(B_k), (C_k)$ of measurable subsets of $X$
such that, for all $k\geq k_0$,
$$
\mu(A\backslash B_k)<\epsilon, \,\, \mu(f^{-k}(B_k))<\epsilon, \,\, f^k(B_k)\subset C_k\textrm{  and  } \mu(C_k)<\epsilon.
$$
\end{enumerate}
\end{theorem}
Conditions (A2) and (B2) are simplified when we add extra-assumptions on $f$ or $X$. For instance, if $f$ is one-to-one and bimeasurable, 
the condition $f^{-1}(\mathcal B)=\mathcal B$ is automatically satisfied
and we do not need the set $C$ anymore since we know that $f^{k}(B)$ belongs to $\mathcal B$.
If we assume moreover that the measure $\mu$ is finite, then Theorem \ref{thm1} simplifies into

\begin{corollary}\label{cthm1} 
Let $(X,\mathcal B,\mu)$ be a finite measure space and $\mathcal X$ be an admissible Banach space of functions defined on $X$.
Let $f:X\to X$ be one-to-one, bimeasurable and such that the composition operator $T_f:\varphi\mapsto\varphi\circ f$ is bounded on $\mathcal X$.
 Then the following assumptions are equivalent:
%Let $X=(X,\mathcal{B},\mu)$ be a $\sigma$-finite measure space and $f:X\to X$ be an one-to-one bimeasurable transformation  satisfying $(\ref{cond1})$, then each of the following is a sufficient condition for $T_f$ to be topologically transitive:
\begin{itemize}
\item [(C1)] $T_f$ is topologically transitive;
\item [(C2)] For each $\epsilon>0$, there exist a measurable set $B$ and $k\ge 1$ such that
$$
\mu(X{\setminus} B)< \epsilon\quad \textrm{and}\quad B\cap f^k(B)=\emptyset.$$
\end{itemize}  
%  Moreover, if the measure $\mu$ is finite, then Conditions $(C1)$-$(C3)$ are also necessary for $T_f$ to be topologically transitive.
 \end{corollary}

%  Conditions (C1)-(C3) are mutually equivalent. They are not necessary for $T_f$ to be topologically transitive when $\mu(X)=\infty$.
Condition (C2) is the measurable version of the {\it run-away property} that appears in the study of  hypercyclic composition operators acting on the space of smooth functions.
When the measure $\mu$ is $\sigma$-finite and $\mu(X)=\infty$, Condition (C2) is sufficient but not always necessary for the topological transitivity of $T_f$.

The proofs of Theorems \ref{thm1}, \ref{thm2} and Corollary \ref{cthm1} are given
 in Section \ref{proofsmain}.
Section \ref{sec:examples} is devoted to examples.
The first one is very classical: it is the weighted backward shift on weighted $\ell^p$-spaces. As already observed by Kalmes, the famous theorem of Salas \cite{S1}
falls into our context. We nevertheless recall the result to give an example of a topologically transitive composition operator on a space with infinite measure which does 
not satisfy Condition (C2). We then study classical classes of composition operators, like nonsingular odometers or those induced by inner functions. We also exhibit several
more specific examples pointing out the relevance of our assumptions. Finally, Section \ref{bili} is dedicated to composition operators induced by bi-Lipschitz $\mu$-transformations.

\section{Proofs of the main theorems and its corollaries}\label{proofsmain}

\begin{proof}[Proof of Theorem \ref{thm1}]  
We first show that (A1)$\Rightarrow$(A2). First of all, if $T_f$ is topologically transitive, then it has dense range and the condition $f^{-1}(\mathcal B)=\mathcal B$ characterizes composition
operators $T_f$ with dense range (see \cite[Corollary 2.2.8]{RM}). Given $\epsilon>0$, let $\delta>0$ be such that the claim of (H3)  holds.
Let $A\in\mathcal B$ have finite measure, then, by (H1), $\rchi_A\in \mathcal X$. Since $T_f$ is topologically transitive, there exist $k\geq 1$ and $\varphi\in \mathcal X$ such that
\begin{equation}\label{deltadelta}
\left\Vert\varphi-2\rchi_A\right\Vert\leq \delta\quad\textrm{and}\quad\left\Vert\varphi\circ f^k-4\rchi_A\right\Vert\le\delta.
\end{equation}
 Let us denote
\begin{equation*}
C=\big\{x\in X:\ |\varphi(x)-4|<1\big\}\quad\textrm{and}\quad D=\big\{x\in A:\ |\varphi(x)-2|<1\big\}.
\end{equation*}
By $(\ref{deltadelta})$ and (H3) applied to $\psi=\varphi-2\rchi_A$ and $S=A{\setminus} D$, we have  $\mu(A\backslash D)<\epsilon/2$. The same arguments applied 
to $\psi=\varphi\circ f^k-4\rchi_A$ and $S=A{\setminus} f^{-k}(C)$ yield $\mu(A\backslash f^{-k}(C))<\epsilon/2$.
Hence, denoting $B=D\cap f^{-k}(C)$,
we have $\mu(A\backslash B)<\epsilon$ as required.

Concerning the measure of $S=f^{-k}(B)$, observe that $\left\vert\varphi\circ f^k-2\right\vert<1$ on $S$. Hence, the function $\psi=\varphi\circ f^k - 4 \rchi_A$ satisfies $\left\vert \psi\right\vert\ge 1$ on $S$. Then, $(\ref{deltadelta})$ and (H3) leads to $\mu\left(f^{-k}(B)\right)<\epsilon$ as required.

Finally, by definition of $B$, we have that $f^k(B)\subset C$. Moreover, since $C\subset A\backslash D$, we conclude that $\mu(C)<\epsilon$.
 
Now let us prove that (A2)$\Rightarrow $(A1). Let $U,V$ be nonempty open subsets of $\mathcal X$. There exist $A\in\mathcal B$ with finite measure,
simple functions $\psi_1=\sum_i a_i \rchi_{A_i}$, $\psi_2=\sum_j b_j \rchi_{B_j}$ with $A_i,B_j\subset A$ and $\eta>0$ 
such that $B(\psi_1,\eta)\subset U$ and $B(\psi_2,\eta)\subset V$. Without loss of generality, we may assume that the sets $B_j$ are pairwise disjoint and 
we set $M=\max(\|\psi_1\|_\infty,\|\psi_2\|_\infty)$.
Let $\epsilon>0$ be such that the claim of (H4) holds. Let $B,C$ and $k$ be given by (A2). We first define
$\gamma_1=\sum_i a_i \rchi_{A_i\cap B}$ and $\gamma_2=\sum_j b_j\rchi_{B_j\cap B}$. Let $S=A\backslash B$, then $\mu(S)<\epsilon$ by (A2)
and $\gamma_{\ell}-\psi_{\ell}=0$ on $X\backslash S$ for all $1\le\ell\le 2$. Moreover, as $\left\vert \gamma_{\ell}-\psi_{\ell}\right\vert\le 2M$,
 (H4) yields
\begin{equation}\label{aux10}
\left\Vert \gamma_{1}-\psi_{1}\right\Vert<\dfrac{\eta}{2}\quad\textrm{and} \quad \left\Vert \gamma_{2}-\psi_{2}\right\Vert<\dfrac{\eta}{2}.
\end{equation}

Since $f^{-1}(\mathcal B)=\mathcal B$, we also have $f^{-k}(\mathcal B)=\mathcal B$
so that, for all $j$, there exist $C_j\in\mathcal B$ satisfying $f^{-k}(C_j)=B_j\cap B$.

We then define $\varphi\in \mathcal X$ by 
$$\varphi=\left\{
\begin{array}{ll}
 \gamma_1&\textrm{ on } B\backslash \bigcup_j C_j\\
 b_j&\textrm{ on each } C_j\\
 0&\textrm{ outside }B\cup\bigcup_j C_j
\end{array}
\right.$$
and we claim that $\varphi\in U$ whereas $\varphi\circ f^k\in V$. Indeed, 
$\varphi-\gamma_1=0$ except eventually on $\bigcup_j C_j\subset C$, where $|\varphi-\gamma_1|\leq 2M$. By (A2), $\mu(C)<\epsilon$. Applying (H4) we conclude that
$\|\varphi-\gamma_1\|<\eta/2$, which together with $(\ref{aux10})$ implies $\varphi\in U$.
In the same way, $\varphi\circ f^k-\gamma_2=0$ except eventually on $f^{-k}(B)$, where $\left\vert\varphi\circ f^k-\gamma_2\right\vert\leq 2M$. By (A2), $\mu\left(f^{-k}(B)\right)<\epsilon$. Applying (H4) yields
$\left\Vert\varphi\circ f^{-k}-\gamma_2\right\Vert<\eta/2$, which together with $(\ref{aux10})$ implies $\varphi\circ f^{-k}\in V$.
\end{proof}

The proof of Theorem \ref{thm1} also works as a proof of Theorem \ref{thm2}.

\begin{proof}[Proof of Corollary \ref{cthm1}] 
We first show that (C2)$\Rightarrow$ (A2). Note that $f^{-1}(\mathcal{B})=\mathcal{B}$ because $f$ is one-to-one and bimeasurable. Now pick $\epsilon>0$ and $A\in\mathcal B$ with finite measure. Let $B\in\mathcal B$ be such that $\mu(X\backslash B)<\epsilon$ and $B\cap f^k(B)=\emptyset$, then $\mu(A\backslash B)\leq \mu(X\backslash B)<\epsilon$
and $\mu\big(f^k(B)\big)\leq \mu(X\backslash B)<\epsilon$. The last inequality $\mu\big(f^{-k}(B)\big)<\epsilon$ follows from the equivalence
$$B\cap f^k(B)=\emptyset\iff f^{-k}(B)\cap B=\emptyset.$$ Hence, (C2)$\Rightarrow$ (A2) and by Theorem \ref{thm1}, (C2)$\Rightarrow$(C1). To prove the converse, suppose that $T_f$ is topologically transitive. Let $\epsilon>0$ and $A=X$, then by (A2) there exist $B'\in\mathcal{B}$ and $k\ge 1$ such that 
\begin{equation}\label{aux78}
\mu(X{\setminus} B')<\frac{\epsilon}{2}\quad\textrm{and}\quad  \mu\left( f^k(B')\right)<\frac{\epsilon}{2}.
\end{equation}
Set $B=B'{\setminus} f^k(B')$, then $(\ref{aux78})$ implies (C2) because
$$ \mu (X{\setminus} B)\le \mu  (X{\setminus} B') + \mu \left( f^k(B)\right)<\epsilon\quad\textrm{and}\quad B\cap f^k(B)=\emptyset.
$$

\end{proof}

\begin{remark}
 Upon the assumptions of Corollary \ref{cthm1}, Condition (C2) is also equivalent to any of the following conditions.
 \begin{itemize}
\item [(C3)] For each $\epsilon>0$, there exists a measurable set $B$ such that
$$ \mu(X{\setminus} B)< \epsilon\quad\textrm{and}\quad \liminf_{k\to\infty} \mu\left(B\cap f^k(B)\right)=0.
$$
\item [(C4)] For each $\epsilon>0$, there exists a measurable set $B$ such that
$$ \mu(X{\setminus} B)< \epsilon\quad\textrm{and}\quad \liminf_{k\to\infty} \mu\left( f^k(B)\right)=0.
$$
 \end{itemize}  
\end{remark}
 Condition (C3) implies that the underlying map $f$ is not light mixing (see \cite[p. 10]{KP}). Condition (C4) 
 implies that $\mu$ is neither invariant nor equivalent to an invariant finite measure (see \cite[p. 5]{EHIP}).

\begin{proof}

\noindent (C2)$\Rightarrow$(C3). Let $0<\epsilon<\frac{\mu(X)}{2}$. By (C2), for each $k\ge 1$, there exist $B_k\in\mathcal{B}$ and $n_k\ge 1$ such that
	\begin{equation*}
	\mu(X{\setminus} B_k)< \dfrac{
	\epsilon}{2^k}\quad\textrm{and}\quad  B_k\cap f^{n_k}(B_k)=\emptyset.
	\end{equation*}
Set $B=\cap_{k\ge 1} B_k$, then for every $k\ge 1$,
\begin{equation}\label{fr}
\mu\left(X{\setminus} B\right)<\epsilon\quad\textrm{and}\quad
\mu\left(f^{n_k}(B)\right)<\frac{\epsilon}{2^k}.
\end{equation}
We claim that $\{n_k:k\ge 1\}$ is an infinite set. By way of  contradiction, assume  that there exist $m\ge 1$ and infinitely many $k$'s such that 
$n_k=m$. Hence, by the second inequality in $(\ref{fr})$, we have that
  $\mu\left(f^m(B)\right)<\frac{\epsilon}{2^k}$ for infinitely many $k$'s, leading to $\mu\left(f^m(B)\right)=0$. On the other hand,
since $f$, hence $f^m$, is nonsingular, this implies $\mu(B)=0$, a contradiction with the first inequality in $(\ref{fr})$.
%   Condition $(\ref{cond1})$,
% the first inequality in $(\ref{fr})$ and the hypothesis that $\epsilon<\frac{\mu(X)}{2}$ yield
% $ \mu\left(f^m(B)\right)\ge c_1^m \mu(B)>0,
% $     
% which is a contradiction. 
This proves the claim. The claim together with $(\ref{fr})$ yield $ \liminf_{k\to\infty} \mu\left(B\cap f^k(B)\right)=\liminf_{k\to\infty} \mu\left(f^k(B)\right)=0, $ showing that Condition (C3) holds.

\noindent (C3)$\Rightarrow$(C4). Let $\epsilon>0$. By (C3), for each $k\ge 1$, there exist $B_k\in\mathcal{B}$ and $n_k\ge k$ such that
	\begin{equation}\label{fr2}
	\mu(X{\setminus} B_k)< \dfrac{
	\epsilon}{2^k}\quad\textrm{and}\quad \mu\left(f^{n_k}(B_k)\cap B_k\right)< \dfrac{\epsilon}{2^k}.
	\end{equation}
Set $B=\cap_{k\ge 1} B_k$, then $(\ref{fr2})$ leads to $\mu\left(X{\setminus} B\right)=\mu\left(\cup_{k\ge 1} X{\setminus} B_k\right)<\epsilon$
and 
$$\mu\big(f^{n_k}(B)\big)\leq \mu\big(f^{n_k}(B_k)\big)\leq \mu\big(f^{n_k}(B_k)\cap B_k\big)+\mu(X\backslash B_k)<\frac{\epsilon}{2^{k-1}}.$$

\noindent (C4)$\Rightarrow$(C2). Let $\epsilon>0$. By (C4), there exist $B'\in \mathcal{B}$ and  $k\ge 1$ such that
$(\ref{aux78})$, and therefore (C2), holds with $B=B'{\setminus} f^k(B')$.
\end{proof}

It is possible to give a similar corollary for topologically mixing maps. We omit the proof.

\begin{corollary}\label{cthm2} Let $(X,\mathcal B,\mu)$ be a finite measure space and $\mathcal X$ be an admissible Banach space of functions defined on $X$.
Let $f:X\to X$ be one-to-one, bimeasurable and such that the composition operator $T_f:\varphi\mapsto\varphi\circ f$ is bounded on $\mathcal X$.
 Then the following assumptions are equivalent:
\begin{itemize}
\item [(D1)] $T_f$ is topologically mixing.
\item [(D2)] For each $\epsilon>0$, there exist $k_0\ge 1$ and measurable sets $\{B_k\}_{k=k_0}^\infty$ such that 
$$ \mu(X{\setminus} B_k)< \epsilon\quad\textrm{and}\quad B_k\cap f^k(B_k)=\emptyset\quad\textrm{for every}\quad k\ge k_0;
$$
\item [(D3)] For each $\epsilon>0$, there exist  measurable sets $\{B_k\}_{k=1}^\infty$ such that
$$ \mu(X{\setminus} B_k)< \epsilon\quad\textrm{and}\quad \lim_{k\to\infty}\mu\left(B_k\cap f^k(B_k)\right)=0;
$$
\item [(D4)] For each $\epsilon>0$, there exist measurable  sets $\{B_k\}_{k=1}^\infty$ such that
$$ \mu(X{\setminus} B_k)< \epsilon\quad\textrm{and}\quad \limsup_{k\to\infty}\mu\left(f^k(B_k)\right)<\epsilon.
$$
\end{itemize}
When the measure $\mu$ is $\sigma$-finite and $\mu(X)=\infty$, then any of the conditions (D2)-(D4) is sufficient but not always necessary for the topological mixing of $T_f$.
\end{corollary}

\section{Examples}\label{sec:examples}
\subsection{Backward shift operators}  Well-studied classes of operators such as Rolewicz operators, weighted shift operators (see \cite{S1}) and the recently introduced
Rolewicz-type operators (see \cite{BonDarPia}) can be put into our framework. 

The backward shift on weighted $\ell^p$-spaces is the most natural example of composition operator. Let $X=\mathbb N$ or $X=\mathbb Z$ and $\nu=(\nu_i)_{i\in X}$
be a sequence of positive real numbers. Let $\sigma:X\to X,\ i\mapsto i+1$. The composition operator $T_\sigma$ is called a unilateral backward
shift if $X=\mathbb N$ and a bilateral backward shift if $X=\mathbb Z$. It is well-known that $T_{\sigma}$ is bounded on $\ell^p(X,\nu)=\{x\in \mathbb{C}^X;\ \|x\|^p=\sum_{i\in X}|x_i|^p\nu_i<\infty\}$ if and only if 
$\sup_{i\in X}\frac{\nu_i}{\nu_{i+1}}<\infty$.

The topological transitivity and topological mixing of $T_\sigma$ was characterized by Salas in \cite{S1} (cf. Theorem \ref{thm:bs} below). By proceeding as in Kalmes \cite[Theorem 2.8]{Kal07}, identifying $\ell^p(X,\nu)$ with $L^p(X,\mu)$,  where $ \mu=\sum_{i\in X}\nu_i\delta_i$, one can realize that the main theorem of \cite{S1} 
is a consequence of Theorems \ref{thm1} and \ref{thm2}. 

\begin{theorem}[Salas \cite{S1}]\label{thm:bs}
 Let $p\in[1,\infty)$, $X=\N$ or $X=\Z$, $\{\nu_i\}_{i\in X}$ be a sequence
	of positive real numbers such that $\sup_{i\in X}\frac{\nu_i}{\nu_{i+1}}<\infty$
	 and $\sigma:X\to X$ be the map defined by $\sigma(i)=i+1$. Then the following holds
	 	 \begin{itemize}
	 \item [$(a)$] In the case that $X =\N$, $T_\sigma$ is topologically transitive on $\ell^p(X,\nu)$ iff  $\liminf_{i\to \infty} \nu_i = 0$.
	 	 \item [$(b)$] In the case that $X = \Z$, $T_\sigma$ is topologically transitive on $\ell^p(X,\nu)$ iff there exists an increasing sequence of positive integers $(n_k)_{k\in\N}$ such that, for every $i \in \Z$, 
	 	 \[ \lim_{k \rightarrow \infty} {\nu_{i+n_k}}= 0 \ \ \textrm{and}  \ \  \lim_{k \rightarrow \infty} {\nu_{i-n_k}} = 0. \]
	 	\item [$(c)$] $T_{\sigma}$ is topologically mixing on $\ell^p(X,\nu)$ iff $\lim_{| i | \to \infty}  \nu_i=0.$
	  \end{itemize}
	 In particular, if $X$ is a finite measure space $\big($i.e. $\sum_{i\in X} \nu_i<\infty\big)$, then $T_{\sigma}$ is topologically mixing. 	
\end{theorem}
An interesting feature of  Theorem \ref{thm:bs} is that it allows us to exhibit an example of a topologically transitive composition operator which does not satisfy (C2),
showing that we cannot dispense with the assumption that $X$ has finite measure in Corollary \ref{cthm1}.

\begin{corollary}\label{cor32}
There exists a $\sigma$-finite measure space $(\N, 2^{\N},\mu)$ such that the map $\sigma:\N\to \N$  defined by $\sigma(i)=i+1$ is one-to-one, bimeasurable and its composition operator
$T_{\sigma}$ is topologically transitive but not topologically mixing. Moreover, $\sigma$ does not satisfy Condition (C2) of Corollary \ref{cthm1}.

\end{corollary}
		\begin{proof}
	Let $(\nu_i)_{i\in\N}$ be the sequence of positive real numbers
	\begin{equation}\label{sequence}
	\underbrace{1, \dfrac{1}{2}}_{V_1}, \underbrace{1,\dfrac{1}{2},\dfrac{1}{4}, \dfrac{1}{2}}_{V_2},\underbrace{1,\dfrac{1}{2},\dfrac{1}{4},\dfrac{1}{8},\dfrac{1}{4},\dfrac{1}{2}}_{V_3}, \ldots, \underbrace{1,\dfrac{1}{2},\ldots,\dfrac{1}{2^n},\dfrac{1}{2^{n-1}},\ldots \dfrac{1}{2}}_{V_n, \,\, n\ge 3}, \ldots
	\end{equation}
	Note that the general term $\nu_{i+1}$ is obtained from its antecessor $\nu_i$ by multiplication by $2$ or division by $2$, that is, 
	$$ \dfrac{1}{2}\le \dfrac{\nu_{i+1}}{\nu_{i}}\le 2,\quad \forall i\in\N.
	$$
	Hence, $T_\sigma$ is an invertible bounded linear operator from $L^p \big(X,\mu \big)$ into itself, where $\mu=\sum_i \nu_i \delta_i$. By $(\ref{sequence})$,  
	$ \liminf_{i\to\infty} \nu_i=0$. By the item (a) of Theorem \ref{thm:bs}, $T_{\sigma}$ is
	topologically transitive. On the other hand, we have that
	$\limsup_{i\to\infty} \nu_i=1$. Hence, by the item (c) of the same theorem, $T_{\sigma}$ is not topologically mixing.
	
	Now we will prove that $\sigma$ does not satisfy Condition (C2) in the statement of Corollary \ref{cthm1}. Specifically, let $B \subset \N$ be such that $\N{\setminus} B$ has finite measure. We will show that for all $k \in \N$, $B\cap \sigma^k(B) $ has infinite measure. 
	
	Let $V_1,V_2,\ldots$ be the blocks of real numbers defined in (\ref{sequence}). Denote by $| V_n|$ the cardinality of the set $V_n$. Let $\{I_n\}_{n\in\N}$ be the partition of $\N$ defined recursively by $I_1=\{1,2\}$ and
	$I_{n+1}=\{1+\max I_n,\ldots,|V_{n+1}|+\max I_n\}$, $n\ge 1$, that is,
	$$I_1=\{1,2\}, \, I_2=\{3,4,5,6\}, \, I_3=\{7,8,9,10,11,12\}, \ldots$$ 
	Let $\{c_n\}_{n\in \N}$ be the increasing sequence of positive integers defined  by $c_n=\min I_n$, then $\nu_{c_n}=1$ for every $n\in\N$. As $\mu (\N{\setminus} B) < \infty$, we have that $c_n\in B$ for all but finitely many $n$'s. Let $m_0\in\N$ be such that $c_n\in B$ for all $n\ge m_0$. For each $n \ge m_0$,  let $d_n \in I_n$ be the largest positive integer such that $[c_n,d_n] \subset B$. We claim that $\lim_{n\rightarrow \infty} (d_n -c_n) = \infty$. In fact, if this is false, then there is an increasing sequence $\{e_n\}_{n\in\N}\subset \N{\setminus B}$ such that $\mu(e_n)$ is a positive constant of the form $1/2^{\ell}$, contradicting the fact that $\mu (\N{\setminus} B) < \infty$. Fix $k \in \N$. Let $m_k \ge m_0$ be such that $d_n-c_n >2k$ for all $n \ge m_k$. Note that for any such $n$, we have that $\sigma^k(c_n) = c_n +k < d_n$ and hence $\sigma^k(c_n) \in B$. In this way, $\sigma^k(c_n) \in B \cap \sigma^k(B)$ for every $n\ge m_k$. Moreover, since $\sigma^k(c_n)=c_n+k<(c_n+d_n)/2$, we have that $\mu(\sigma^k(c_n)) = 1/2^k$. Hence we reach
	 \[\mu (B \cap f^k(B)) \ge  \sum_{n=m_k}^{\infty} \mu(\sigma^k(c_n)) = \sum_{n=m_k}^{\infty} \dfrac{1}{2^k} = \infty,  \]
	 proving that $\sigma$ does not satisfy (C2).
	 	  \end{proof}

\subsection{Composition by inner functions}

Let $\mathbb D$ be the complex unit disc and $\mathbb T=\partial \mathbb D$ be its boundary, the unit circle. 
We recall that the classical Fatou's Theorem asserts that a bounded holomorphic function $f:\mathbb D\to\mathbb C$
has radial limits almost everywhere. A holomorphic function $f:\mathbb D\to\mathbb D$ is called an \emph{inner} function
if the radial limits
$$f^*(\xi)=\lim_{r\to 1^-}f(r\xi)$$
have modulus $1$ for almost every $\xi\in\mathbb T$. Therefore, if $f$ is inner, the radial limits define a map $f^*:\mathbb T\to\mathbb T$ 
up to a set of zero Lebesgue measure. In what follows, we will simply denote by $f$ the function $f^*$.

It was proved by Nordgren \cite{Nor68} that an inner function $f$ induces a bounded composition operator $T_f$ on $L^2(\mathbb T,d\lambda)$ 
where $\lambda$ denotes the normalized Lebesgue measure on $\mathbb T$. The key point in Nordgren's argument is the following fact: for $\varphi\in L^1(\mathbb T)$,
denote by $P[\varphi]$ the Poisson integral of $\varphi$ defined on $\mathbb D$ by
$$P[\varphi](z)=\int_{\mathbb T}\mathrm{Re}\left(\frac{\xi+z}{\xi-z}\right)\varphi(\xi)d\lambda(\xi).$$
Then, for any $\varphi\in L^1(\mathbb T)$ and any inner function $f$, the relation 
$$P[\varphi\circ f]=P[\varphi]\circ f$$
holds. This lemma may be interpreted as a result about how inner functions transform the Poisson measures on $\mathbb T$
$$dm_\omega(\xi)=\mathrm{Re}\left(\frac{\xi+\omega}{\xi-\omega}\right)d\lambda(\xi),$$
where $\omega\in\mathbb D$. It says that $dm_\omega\circ f^{-1}=dm_{f(\omega)}$. 
Applying this for the particular case $\omega=0$, we thus get
$$\frac{1-|f(0)|}{1+|f(0)|}\|\varphi\|_2^2\leq \|T_f(\varphi)\|_2^2\leq \frac{1+|f(0)|}{1-|f(0)|}\|\varphi\|_2^2.$$

We are interested in the topological transitivity of $T_f$.

\begin{theorem}
 Let $f$ be an inner function. The following conditions are equivalent:
 \begin{itemize}
  \item[(i)] $T_f$ is topologically transitive on $L^2(\mathbb T)$.
  \item[(ii)] $T_f$ is topologically mixing on $L^2(\mathbb T)$.
  \item[(iii)] $f$ is an automorphism of the disk with no fixed point in $\mathbb D$.
 \end{itemize}
\end{theorem}
\begin{proof}
 We first assume that $T_f$ is topologically transitive. Let us show that $f$ has no fixed point in $\mathbb D$. We argue by contradiction and assume that
 it has a fixed point $\omega\in\mathbb D$. By the discussion above, the associated Poisson measure $m_\omega$ is invariant under the action of $f$.
 Observe also that there exist two constants $c_1,c_2>0$ such that $c_1 m_\omega(B)\leq \lambda(B)\leq c_2 m_\omega(B)$
 for all measurable sets $B\subset\mathbb T$. 
 
Let now $\epsilon\in (0,1)$ and $B\subset\mathbb T$ be measurable such that $\lambda(\mathbb T\backslash B)<\epsilon$.
Then 
$$\lambda\big(f^{-k}(B)\big)\geq c_1 m_\omega\big(f^{-k}(B)\big)= c_1 m_\omega(B)\geq\frac{c_1}{c_2}\lambda(B),$$
so that $\lambda\big(f^{-k}(B)\big)$ cannot be smaller than $\epsilon$ if $c_1(1-\epsilon)\geq c_2\epsilon$. This contradicts (A2).

We now show that, provided $T_f$ is topologically transitive, then $f$ is an automorphism of $\mathbb D$. Indeed, the composition operator associated to an inner function has closed range
(see the proof of Theorem 2.2.7 in \cite{RM} where it is shown that a composition operator $T_f$ acting on $L^2$ has closed range if and only if $f$ is essentially
bounded away from zero). If
we moreover assume that $T_f$ has dense range, then we get that $T_f$ is invertible on $L^2(\mathbb T)$. By \cite[Theorem 2.2.14]{RM}, there exists a measurable map $g:\mathbb T\to\mathbb T$
such that, for almost all $\xi\in\mathbb T$, $g\circ f(\xi)=\xi$. Taking the Poisson transform of this equality yields $P[g]\circ f(z)=z$ for all $z\in\mathbb D$. 
In particular, $f$ has to be injective on $\mathbb D$. It is well-known that this implies that $f$ must be an automorphism of $\mathbb D$. For instance, by Frostman's Theorem,
there exists $\xi\in\mathbb T$ such that $f_\xi(z)=\frac{f(z)-\xi}{1-\bar\xi f(z)}$ is a Blaschke product, and $f_\xi$ remains one-to-one, being the composition of an automorphism
of $\mathbb D$ and an injective map. Thus, $f_\xi$ is a Blaschke product of degree 1, namely an automorphism, and $f$ itself is also an automorphism.

Conversely, assume that $f$ is an automorphism of $\mathbb D$ which has no fixed point in $\mathbb D$. For $\xi\in\mathbb T$, we denote by $I(\xi,\epsilon)$ the arc of $\mathbb T$ with Lebesgue measure $\epsilon$ 
and midpoint $\xi$. Observe also that $f^{-1}$ is also an automorphism of $\mathbb D$ without fixed point in $\mathbb D$. There exist $\omega_1,\omega_2$ in $\mathbb T$ (the respective Denjoy-Wolff
points of $f$ and $f^{-1}$) such that $f^k\to\omega_1$ uniformly on all $\mathbb T\backslash I(\omega_2,\epsilon)$ and 
$f^{-k}\to\omega_2$ uniformly on all $\mathbb T\backslash I(\omega_1,\epsilon)$, $\epsilon>0$
(we may have $\omega_1=\omega_2$; this happens if and only if $f$ is a parabolic automorphism of $\mathbb D$). Fix now $\epsilon>0$ and let 
$B=\mathbb T\backslash \big(I(\omega_1,\epsilon/2)\cup I(\omega_2,\epsilon/2)\big)$. Then $\lambda(B)\geq 1-\epsilon$. Moreover, $\lambda(f^{-k}(B))=\int_{\mathbb T}\rchi_B\circ f^kd\lambda$. Since $\rchi_B\circ f^k$
converges uniformly to 0 on $B$ and since $\lambda(\mathbb T\backslash B)\leq \epsilon$, we get that for $k$ large enough, $\lambda(f^{-k}(B))\leq\epsilon$. Similarly,
also for $k$ large enough, $\lambda(f^k(B))\leq \epsilon$. This ensures that $T_f$ is topologically mixing.

\end{proof}

\begin{remark}
 If $f$ is an inner function, $T_f$ induces also a bounded composition operator on the Hardy space $H^2(\mathbb D)$. It can be easily deduced from existing results in the literature that $T_f$
 is topologically transitive on $L^2(\mathbb T)$ iff $f$ is an automorphism of the disk without fixed point in $\mathbb D$. Indeed, that this condition is sufficient appears in \cite{JHS}.
 Conversely, if $T_f$ is topologically transitive, then $f$ has no fixed points in $\mathbb D$ (see \cite[Prop. 1.45]{BM}) and it is necessarily one-to-one (otherwise, it would not have
 dense range since the evaluation at a point $a\in\mathbb D$ is continuous in $H^2(\mathbb D)$). We conclude exactly as above.
\end{remark}

\subsection{Nonsingular odometer}\label{sodometer}

In this section, we provide an example of a bijective bimeasurable transformation \mbox{$f:X\to X$} defined on a probability space
$X=(X,\mathcal{B},\mu)$ whose composition operator $T_f$ acting on $L^p(X,\mathcal{B},\mu)$, $1\le p <\infty$,
 is topologically transitive but not topologically mixing (the example of Corollary \ref{cor32} was defined on a space with infinite measure). This class of transformations is known in ergodic theory as {\it odometer} or {\it adding machine} \big(see \cite{DaSi} for more information on this topic\big). Now we proceed with the definitions of $X$ and $f$. 
 
 The set $X$ is defined by
  \begin{equation}\label{defX} 
  X=\prod_{i= 1}^\infty A_i,\quad \textrm{where} \quad A_i=\begin{cases} 
\{0,1\} & \textrm{if $i$ is even} 
 \\
 \{0,1, \ldots, 2i-1\}  & \textrm{if $i$ is odd} \end{cases}.
 \end{equation}
We endow $A_i$ with the discrete topology and with the purely atomic probability measure \mbox{$\mu_i:2^{A_i}\to [0,1]$} which assigns to the atom $\{j\}$ of $A_i$ the value $\mu_i(j)$ defined by
\begin{equation}\label{even}
\mu_i(0)=\mu_i(1)=\dfrac12,\quad \textrm{if $i$ is even}
\end{equation}
and
\begin{equation}\label{muij}
\mu_i(j)=\begin{cases}
\dfrac{1-2^{-i}}{i} & \textrm{for $j\in \{0,\ldots,i-1\}$}\\[0.2in]
\phantom{aa}\dfrac{2^{-i}}{i} & \textrm{for $j\in \{i,\ldots,2i-1\}$}
\end{cases},\quad \textrm{if $i$ is odd.}
\end{equation}
We denote by $\mathcal{B}$ the product $\sigma$-algebra and by $\mu$ the product measure on $X=\prod_{i=1}^\infty A_i$. In this way, $X=(X,\mathcal{B},\mu)$ becomes a probability space. A point of $X$ is an infinite sequence $x=(x_i)_{i=1}^\infty$ with $x_i\in A_i$ for each $i\ge 1$.

The map $f:X\to X$ is defined as follows. If $x\in X$ is the point
$$ x^*=(\max A_i)_{i=1}^\infty=(1,1,5,1,9,\ldots)
$$ 
then $f$ assigns to it the infinite sequence $f(x^*)=(0,0,0,0,0,\ldots)$. Otherwise, $x\neq x^*$, thus there exists $i\ge 1$ such that
 $x_{i}<\max A_{i}$. In this case, let $\ell(x)=\min\, \{i\ge 1 ;\  x_i < \max A_i\}$ be the least positive integer $i$ with such property and set $f(x)$ to be the infinite sequence whose $i$-th term $f(x)_i$ is
 \begin{equation}\label{fxi}
 f(x)_i=
 \begin{cases} 
 0 & \textrm{if } i< \ell(x) \\
  x_i+1 & \textrm{if } i=\ell(x) \\
  x_{i} & \textrm{if } i>\ell(x).
 \end{cases}.
 \end{equation}
 
 The system $(X,f)$ is sometimes called {\it adding machine} or {\it odometer}. It is well-known that $f$ is a homeomorphism of $X$. The other properties of the system $(X,f)$ are provided in the claims below. \\

\noindent Claim A. $T_f$ is a bounded linear operator acting on $L^p(X,\mathcal{B},\mu), \,1\le p<\infty$.

\begin{proof} Let $x=(x_i)_{i=1}^\infty\in X$. For each $n\ge 1$, let
 $$[x_1,\ldots,x_n]=\{y=(y_i)_{i=1}^\infty\in X ;\  y_1=x_1,\ldots, y_n=x_n\}.$$
Denote by $\mu\circ f$ the Borel probability measure defined by $\mu\circ f(B)=\mu\left(f(B)\right)$ for every measurable set $B\subset X$. The measure $\mu\circ f$ is absolutely continuous with respect to $\mu$ and its  Radon-Nikodym derivative 
 with respect to $\mu$ at  $\mu$-almost every $x=(x_i)_{i\ge 1}\in X$ takes the value 
 (see \cite[Theorem 7, p. 118]{W} or \cite[Section 3.1]{DaSi})
 \begin{equation}\label{rn1}
 \frac{d\mu\circ f}{d\mu}(x) =\lim_{n\to\infty} \dfrac{\mu \left( f([x_1,\ldots,x_n])\right)}{\mu([x_1,\ldots,x_n])}=\lim_{n\to\infty}\prod_{i=1}^n \dfrac{\mu_i\left( f(x)_i\right)}{\mu_i(x_i)}.
 \end{equation}
 Let $k=\ell(x)$. Suppose first that $k=1$.  By $(\ref{defX})$, $A_1=\{0,1\}$, thus $x_1=0$. Moreover, 
 by $(\ref{fxi})$, we also have $f(x)_1=1$ and $f(x)_i=x_i$ for every $i\ge 2$. Therefore, by $(\ref{rn1})$,  we have that
\begin{equation}\label{muf1}
\frac{d\mu\circ f}{d\mu}(x)=\dfrac{\mu_1\left(f(x)_1\right)}{\mu_1(x_1)}=\dfrac{\mu_1(1)}{\mu_1(0)}=\dfrac{2^{-1}}{1-2^{-1}}=1.
\end{equation}
Assume now that $k\ge 2$, then $(\ref{defX})$, $(\ref{even})$, $(\ref{fxi})$ and $(\ref{rn1})$ lead to
\begin{equation}\label{dmf}
  \frac{d\mu\circ f}{d\mu}(x) =\frac{\mu_{k}(x_{k}+1)}{\mu_{k}(x_{k})}  \prod_{i=1}^{k -1} \frac{\mu_i(0)}{\mu_i(\max{A_i})}=\frac{\mu_{k}(x_{k}+1)}{\mu_{k}(x_{k})}   \prod_{\substack{i=1\\\textrm{$i$ odd}}}^{k -1} \dfrac{1-2^{-i}}{2^{-i}}.
\end{equation}
If $k\ge 2$ is even, then $\mu_k(x_k+1)=\frac12=\mu_k(x_k)$, which together with $(\ref{dmf})$ yields $\frac{d\mu\circ f}{d\mu}(x)\ge 1$. Otherwise, $k\ge 3$ is odd, then $(\ref{muij})$ and $(\ref{dmf})$ implies
\begin{equation*}\label{muf2}
  \frac{d\mu\circ f}{d\mu}(x) \ge  \frac{2^{-k}}{1-2^{-k}}\cdot\frac{1-2^{-(k-2)}}{2^{-(k-2)}}\ge \frac14\cdot \left( 1-\frac{3}{2^k-1}\right)\ge\frac17.
  \end{equation*}
Putting it all together, we obtain that $\frac{d\mu\circ f}{d\mu}(x) \ge \frac{1}{7}$ for $\mu$-almost every $x \in X$. This, in turn, implies that $\mu(f(B)) \ge \frac{1}{7} \mu(B)$ for every measurable
set $B$.
Hence, $T_f$ is a well-defined bounded linear operator (see \cite[Theorem 2.1.1]{RM}).
\end{proof}

\noindent Claim B. $T_f$ is topologically transitive.

\begin{proof}
We will verify Condition (C1) in Corollary \ref{cthm1}. For each $n$ odd, let $A_n=B_n\cup C_n$, where
\begin{equation}\label{BC}
B_n=\{0,\ldots,n-1\}\quad\textrm{and}\quad C_n=\{n,\ldots,2n-1\}.
\end{equation}
Let $\varepsilon >0$ and $n$ be an odd integer  such that $2^{-n} < \varepsilon$. Let 
\[ B =\left\{x=(x_i)_{i=1}^\infty\in X ;\  x_n \in B_n \}\right..
\]
By $(\ref{muij})$, as $n$ is odd, we have that $\mu(B)=\mu_n(B_n)= 1-2^{-n} > 1 - \varepsilon.$ Therefore,
$\mu(X{\setminus} B)<\epsilon$.
Let $k = n \prod_{i=1}^{n-1} |A_i|$, where $|A_i |$ denotes the cardinality of $A_i$. Then, for any $x \in B$, we have that
$$f^k(x)_i = \begin{cases} x_i & \textrm{if $i\neq n$}\\
x_n+n & \textrm{if $i= n$.}\\
\end{cases}
$$
Hence, 
$f^k(B) \subset \left\{x \in X ;\  x_n\in C_n\right\}$, verifying that $B\cap f^k(B) =\emptyset$. 
\end{proof}

\noindent Claim C. $T_f$ is not topologically mixing.

\begin{proof} We will prove that $f$ does not satisfy Condition (D2) in Corollary \ref{cthm2}. This together with the fact that $\mu$ is finite prevents $T_f$ from being mixing. 
	
By way of contradiction, suppose that $f$ satisfies Condition (D2) in Corollary \ref{cthm2}. In particular, for $\epsilon=0.1$, there exist an odd integer
$n\ge 1$ and measurable sets $\{B_k\}_{k=n}^\infty$ such that 
\begin{equation}\label{c1again}
\mu(X{\setminus}B_k)<0.1\quad\textrm{and} \quad B_k\cap f^k(B_k)=\emptyset\quad \textrm{for every}\quad k\ge n. 
\end{equation}
For $i \in A_n$ and $j \in A_{n+1}$, define 
\[ E_i = \{x \in X ;\  x_n =i\} \textrm{ and }   E_{i,j} = \{x \in X ;\  x_n =i  \textrm{ and }  x_{n+1} =j\}.\]
Let $k = \prod_{i=1}^n | A_i |$, where $| A_i | $ denotes the cardinality of $A_i$, then $k\ge n$. By the first inequality in $(\ref{c1again})$, we have that $\mu(B_k) > 0.9$. As $\cup_{i=1}^{2n-1} E_i = X$, there is $i \in A_n$ such that
$\mu (B_k \cap E_i) > 0.9 \mu(E_i)$. As $n+1$ is even, we have by  $(\ref{defX})$ that $A_{n+1}=\{0,1\}$, thus $E_i = E_{i,0} \cup E_{i,1}$ and $\mu(E_{i,0}) = \mu(E_{i,1}) = 0.5\mu(E_i)$, which lead to
\[ \mu(B_k \cap E_{i,0}) > 0.4 \mu(E_i)\quad\textrm{and} \quad  \mu(B_k \cap E_{i,1}) >  0.4\mu(E_i).\]

As $k = \prod_{i=1}^n | A_i | 
$, we have that  if $x \in E_{i,0}$, then $f^k(x)_{j} = x_j$ for all $j \neq (n+1)$ and $f^k(x)_{n+1} = 1$. Hence, $f^k(B_k\cap E_{i,0}) \subset E_{i,1}$. Moreover, as $\mu_{n+1} (0) = \mu_{n+1}(1) =0.5$, we have that $\mu\left(f^k (B_k \cap E_{i,0})\right) = \mu(B_k \cap E_{i,0})$$ > 0.4 \mu(E_i)$.
Hence, we have that $f^k (B_k \cap E_{i,0})$ and $B_k \cap E_{i,1}$ are both subsets of $E_{i,1}$, each having $\mu$-measure at least $0.4 \mu(E_i)$. However, as $E_{i,1}$ has $\mu$-measure $0.5\mu(E_i)$, we have that 
$  (B_k  \cap E_{i,1})\cap f^k (B_k \cap E_{i,0})  \neq \emptyset$. This, in turn, implies that $B_k\cap f^k(B_k) \neq \emptyset$, which contradicts $(\ref{c1again})$.
\end{proof}

\subsection{A topologically transitive composition operator on a non $\sigma$-compact Borel space}\label{s333}
In this subsection, we give an example of a topologically transitive composition operator on a Borel space which is not $\sigma$-compact. 
We start with a Borel probability space $(Y,\nu)$ which is not $\sigma$-compact (for instance, $Y$ could be an infinite dimensional Banach space endowed
with a Gaussian measure). Let $X=[0,1]\times Y$ endowed with the product topology and let $\mu=dx\otimes \nu$. The map $f$ is defined on $[0,1]\times X$ by
$(a,y)\mapsto (a/2,y)$. Then since $\frac 12\mu(B)\leq \mu\big( f(B)\big)\leq 2\mu(B)$, $T_f$ defines a bounded composition operator on $L^1(\mu)$. It satisfies clearly
condition (C2) by choosing $B=[\epsilon,1]\times Y$ and $k\geq 1$ with $2^{-k}<\epsilon$.

\subsection{A topologically mixing composition operator induced by a non-bimeasurable transformation}

We now give an example of a topologically mixing composition operator induced by a non-bimeasurable transformation. 
Let $\mathbb Z$ be endowed with the $\sigma$-algebra $\mathcal B$ generated by the sets $\{k\}$, $k<0$, and the sets $\{2k,2k+1\}$,
$k\geq 0$. Let $f:\mathbb Z\to\mathbb Z$ be defined by 
$$f(n)=\left\{
\begin{array}{ll}
 0&\textrm{if }n=-2\\
 2&\textrm{if }n=-1\\
 n+4&\textrm{if }n\geq 0\\
 n+1&\textrm{if }n\leq -3.
\end{array}\right.$$
Then $f$ is measurable yet not bimeasurable: the image of the measurable set $\{-2\}$ by $f$ is $\{0\}$ which is not measurable.

We endow $\mathcal B$ with the following finite measure $\mu$: for $k<0$, $\mu(\{k\})=2^k$ and for $k\geq 0$, $\mu(\{2k,2k+1\})=2^{-k}$.

\begin{proposition}
 $T_f$ is a topologically mixing transformation of $L^2(\mathbb Z,\mathcal B,\mu)$.
\end{proposition}
\begin{proof}
We first observe that $f^{-1}(\mathcal B)=\mathcal B$. This follows from
$$\left\{
\begin{array}{ll}
 f^{-1}(\{k+1\})=\{k\}&\textrm{if }k\leq -3\\
 f^{-1}(\{0,1\})=\{-2\}\\
 f^{-1}(\{2,3\})=\{-1\}\\
 f^{-1}(\{2k+4,2k+5\})=\{2k,2k+1\}&\textrm{if }k\geq 0.
\end{array}\right.$$

Moreover, let $\epsilon>0$. There exists $n\geq 1$ large enough such that $\mu(C_n)<\epsilon/2$ and $\mu(D_n)<\epsilon/2$
where
$$C_n=\{2n,2n+1,\dots\},\ D_n=\{-2n,-2n-1,\dots\}.$$
We set $B=\mathbb Z\backslash (C_n\cup D_n)$. To conclude, it suffices to observe that for $k$ large enough, $f^k(B)\subset C_n$ and $f^{-k}(B)\subset D_n$.
\end{proof}

\section{Bi-lipschitz $\mu$-transformations}\label{bili}
	
	In this section, we provide a fairly comprehensive study of the composition operator when $f$ is a bi-Lipschitz $\mu$-contraction.
	This study reveals how necessary and how embracing  the hypotheses of our main theorems are.
		
	Throughout this section, we assume that $X=(X,\mathcal{B},\mu)$ is a $\sigma$-finite measure space, $\mathcal{X}\subset L^0(\mu)$ is any admissible Banach space of functions defined on $X$, and $f:X\to X$ is a bimeasurable map such that the composition operator $T_f:\varphi\mapsto\varphi\circ f$ is bounded on $\mathcal X$. 
	
	We say that $f$ is a {\it bi-Lipschitz} $\mu$-{\it transformation} if there exist $0<c_1\le c_2$ such that for every measurable set $B$,
	\begin{equation}\label{blc} 
	c_1 \mu(B)\le \mu\left( f(B)\right)\le c_2\mu (B).
	\end{equation}
Note that if $f$ is a bi-Lipschitz $\mu$-transformation and $\mathcal X=L^p(X,\mathcal{B},\mu)$, then $f$ is nonsingular and the composition operator $T_f$ is bounded on $\mathcal{X}$. If  $f$ satisfies $(\ref{blc})$ with $c_2< 1$, then  $f$ is called  {\it bi-Lipschitz $\mu$-contraction}. 

The next proposition shows that within the category of one-to-one bi-Lipschitz  $\mu$-contractions, topological mixing is equivalent to topological transitivity. 

\begin{proposition}\label{finmes}
Let $(X,\mathcal B,\mu)$ be a $\sigma$-finite measure space and $f:X\to X$ be a one-to-one bi-Lipschitz $\mu$-contraction. Then one of the following happens:
	\begin{itemize}
		\item [$(E1)$] $\mu\left(\cap_{k\ge 1}f^k(X)\right)=0$ and $T_f$ is topologically mixing;
		\item [$(E2)$] $\mu\left(\cap_{k\ge 1}f^k(X)\right)\neq 0$ and $T_f$ is not topologically transitive.
	\end{itemize}
\end{proposition}
\begin{proof}
	Let $\Lambda=\cap_{k\ge 1} f^k(X)$, then $f(\Lambda)=\Lambda$. Suppose  first that  $\mu(\Lambda)=0$. 
% 	Let $\delta>0$, $\psi_1,\psi_2:X\to\mathbb{R}$ be simple functions whose supports  have finite measure and $U_1,U_2\subset L^p(X,\mathcal{B},\mu)$ be open balls of radius $\delta$ centered at  $\psi_1,\psi_2$, respectively. We will prove that there exists $k_0\ge 1$ such that $T_f^k(U_1)\cap U_2\neq\emptyset$ for every $k\ge k_0$. This implies that $T_{f}$ is topologically mixing because the simple functions having supports with finite measure form a dense subset of $L^p(X,\mathcal{B},\mu)$.
	
	Let $W=X{\setminus} f\left(X\right)$, then $W\cap f^k (W)=\emptyset$ and $f^{-k}(W)=\emptyset$ for every $k\ge 1$. This together with the injectivity of $f$ implies that the sets $\{f^k(W), k\in\mathbb{Z}\}$
	are pairwise disjoint. By the injectivity of $f$, for every $k\ge 0$, $f^{k}(W)=f^{k}\left( X\right)\setminus f^{k+1}(X)$, thus
	\begin{equation}\label{eqw1}
	W\cup f(W)\cup \cdots \cup f^{k-1}(W)=X\big\backslash f^{k}\left(X\right)\quad\textrm{for every}\quad k\ge 1.
	\end{equation}
	In this way,  we have that
	\begin{equation}\label{eqw2}
	\bigcup_{k\ge 0} f^k(W)=X\Big\backslash \bigcap_{k\ge 1} f^k(X)=X\setminus \Lambda.
	\end{equation}
	Since $\mu(\Lambda)=0$, we have that $\{f^k(W):k\ge 0\}$ is a partition of $X$ mod $\mu$.
	We shall verify condition (B2) of Theorem \ref{thm2}. Let $A\in\mathcal B$ with finite measure and $\epsilon>0$. We know that $\sum_{k\geq 0}\mu\big(A\cap f^k(W)\big)=\mu(A)$, thus there exists $k_0\ge 0$ such that
	\begin{equation}\label{sum<e}\sum_{k\ge k_0}\mu\big(A\cap f^{k}(W)\big)<\epsilon.
	\end{equation} 
	Set $B_k=\big(W\cup f(W)\cup\cdots\cup f^{k-1}(W)\big)\cap A$. By $(\ref{sum<e})$,
	 $\mu(A\backslash B_k)<\epsilon$ for every $k\ge k_0$. 
	Moreover, $f^{-k}(B_k)=\emptyset$ whereas, by $(\ref{blc})$, $\mu\big( f^k(B_k)\big)\le c_2^k \mu\big(B_k\big)<\epsilon$ for every $k$ big enough, showing that condition (B2) is indeed satisfied.
	
	Now suppose that $\mu(\Lambda)\neq 0$. Let $A\subset\Lambda$ be a set of positive measure. Let $B\subset A$ be such that $\mu(B)\geq\mu(A)/2$. Then,
	since $B\subset A\subset \bigcap_{k\geq 1}f^k(X)$, the injectivity of $f$ ensures that $f^k(f^{-k}(B))=B$. Since $f$ is a  $\mu$-contraction, this 
	implies $\mu\big(f^{-k}(B)\big)\geq \mu(B)\geq \mu(A)/2$ for all $k\geq 1$. Hence, (A2) is not satisfied and $T_f$ is not topologically transitive.
\end{proof} 

\begin{remark}
 It is sufficient to assume that $f$ is a weak $\mu$-contraction (namely to assume that \eqref{blc} with $c_2=1$) to prove that the condition  $\mu\left(\cap_{k\ge 1}f^k(X)\right)\neq 0$ implies 
 that $T_f$ is not topologically transitive.
\end{remark}

	\begin{example}\label{ex33} Let $\mu$ be the Lebesgue measure, then the following assertions hold true:
	\begin{itemize}
	\item [$(i)$] The one-to-one bi-Lipschitz $\mu$-contraction $x\in [0,1]\mapsto \frac{x}{2}$ satisfies (E1);
	\item [$(ii)$] The bijective bi-Lipschitz $\mu$-contraction $x\in\mathbb{R}\mapsto \frac{x}{2}$ satisfies (E2);
	\item [$(iii)$] The one-to-one bi-Lipschitz $\mu$-contraction $(x,y)\in [0,1]\times \mathbb{R}\to \left(\frac{x}{4},2y\right)$ satisfies (E1).
 	\end{itemize}
	\end{example}
		
% 		\begin{remark}\label{rem2} If $f$ is not a bi-Lipschitz $\mu$-contraction, then $T_f$ may  be topologically transitive without being topologically mixing (see Corollary \ref{cor32} or Section \ref{sodometer}). 
% 		\end{remark}
		
		The following corollary follows straightforwardly from Proposition \ref{finmes}.
		
	\begin{corollary}\label{cor}
Let $(X,\mathcal B,\mu)$ be a $\sigma$-finite measure space and $f:X\to X$ be a one-to-one bi-Lipschitz $\mu$-contraction. Then one of the following happens:	\begin{itemize}
	\item [$(a)$] If $\mu$ is finite then $T_f$ is topologically mixing;
         \item [$(b)$] If $f$ is bijective then $T_f$ is not topologically transitive.
	\end{itemize}
	\end{corollary}

	\begin{proof} As $f$ is a $\mu$-contraction, for some $0\le c_2<1$ and every $k\ge 1$, we have that
	$\mu\left(f^k(X)\right)\le c_2^k \mu(X).$
	Hence, if $\mu$ is finite then $\mu\left( \cap_{k\ge 1}f^k(X)\right)=0$ and $T_f$ is topologically mixing by Condition (E1) in Proposition \ref{finmes}.
	On the other hand, if $f$ is bijective, then $\mu(X)=\mu\left(f(X)\right)\le c_2 \mu(X)$, implying $\mu\left( \cap_{k\ge 1}f^k(X)\right)=\mu(X)=\infty$. In this case, by Condition (E2)
	in Proposition \ref{finmes} we have that $T_f$ is not topologically transitive.
			\end{proof}
	
	In finite measure spaces, the hypothesis of $\mu$-contraction is not necessary for topological mixing. Let $(X,\mathcal{B},\mu)$ be a finite measure space and $f:X\to X$ be a bi-Lipschitz $\mu$-transformation. We say that a measurable set $W$ is a {\it wandering set} for $f$ if its iterates $\{f^k(W):k\in\mathbb{Z}\}$ are pairwise disjoint. We say that $W$ is {\it exhaustive} if
	$\mu\left( \cup_{k\in\mathbb{Z}} f^k(W)\right )=\mu(X)$. We say that $f$ is {\it $\mu$-dissipative} if it admits an exhaustive wandering set. The following result generalizes the item (a) of Corollary \ref{cor}. 
	
	\begin{proposition}\label{wanderinginterval} Let $(X,\mathcal{B},\mu)$ be a finite measure space and $f:X\to X$ be a one-to-one $\mu$-dissipative transformation, then $T_f$ is topologically mixing.
	\end{proposition}
	\begin{proof}  The proof consists in verifying Condition (D2) in Corollary \ref{cthm2}. Since $f$ is $\mu$-dissipative, it admits an exhaustive wandering set $W$. Hence, given $\epsilon>0$, there exists $k_0\ge 1$ and
		$B=\cup_{k=-k_0}^{k_0} f^k(W)$ such that $\mu(X{\setminus} B)<\epsilon$. Moreover, for every $k\ge 2k_0+1$, say $k=2k_0+1+p$, where $p\ge 0$, we have that $f^k(B)=\cup_{\ell=k_0+1+p}^{3k_0+1+p} f^{\ell}(W)$, thus $f^k(B)\cap B=\emptyset$.
	\end{proof}
		
		The notion of $\mu$-dissipative transformation is broad in the sense that it includes all the $\mu$-transformations satisfying Condition (E1) in Proposition \ref{finmes}.

		\begin{proposition} Let $(X,\mathcal{B},\mu)$ be a finite measure space and $f:X\to X$ be a one-to-one bi-Lipschitz $\mu$-transformation satisfying $\mu\left(\cap_{k\ge 1} f^k(X)\right)=0$, then $f$ is $\mu$-dissipative.
		\end{proposition}
		\begin{proof} Let $W=X {\setminus} f(X)$. By the proof of Proposition \ref{finmes}, $\{f^k(W):k\ge 0\}$ are pairwise disjoint measurable sets. Moreover, $f^k(W)=\emptyset$ for every $k<0$, which implies that
		$\{f^k(W):k\in\mathbb{Z}\}$ are pairwise disjoint. Moreover, by $(\ref{eqw2})$, $W$ is an exhaustive wandering set, thus $f$ is $\mu$-dissipative.
		\end{proof}
		
		The following example shows that the notion of $\mu$-dissipativity  embraces transformations that are not $\mu$-contractions.
		
		\begin{example} Let $X=[0,1]$ be the unit interval endowed with the Lebesgue measure $\mu$. The following are examples of $\mu$-dissipative transformations that are not $\mu$-contractions:
		\begin{itemize}
		\item [(a)] $f:X\to X$ defined by $f(x)=\log\,(1+x)$.
		\item [(b)] $f:X\to X$ defined by $x\mapsto \frac{x}{2}$ on $\left[0,\frac12\right)$ and $x\mapsto \frac{3}{2}x-\frac{1}{2}$ on $\left[\frac12,1\right)$.
		\end{itemize}
		\end{example}
		\begin{proof} \noindent (a) Let $a_k=f^k(1)$, $k=0,1,2,\ldots$, then $\{a_k\}$ is a strictly decreasing sequence bounded below by $0$. Let $a=\lim_{k\to\infty} a_k$, then, by the continuity of $f$, we have that $a$ is the unique fixed point of $f$, that is, $a=0$. Let $W=[a_1,a_0)$, then $f^k(W)=\left[a_{k+1},a_k\right)$ and hence $W,f(W),\ldots$ are pairwise disjoint sets whose union is $\cup_{k\ge 0} f^k(W)=(0,1)$. Hence, $W$ is an exhaustive wandering set and $f$ is dissipative. Moreover, $f'(0)=1$, which prevents $f$ from being a $\mu$-contraction.\\
 
 \noindent (b) Let $a_k=f^k\left(\frac12\right), k\in\mathbb{Z}$. Then $\{a_k\}_{k=0}^\infty$ is decreasing while $\{a_k\}_{k=0}^{-\infty}$ is increasing. By arguing as in the item (a), we may show that $\lim_{k\to \infty} a_k=0$ and $\lim_{k\to -\infty} a_k=1$. Let $W=\left[\frac12,f\left(\frac12\right)\right)$, then $W$ is an exhaustive wandering set and $f$ is $\mu$-dissipative. Moreover, since $\mu\left( f(X)\right)=\mu(X)=1$, we have that $f$ is not a $\mu$-contraction.
 \end{proof}

		We conclude this paper by showing that the assumption that $f$ is one-to-one in the statement of our results cannot be weakened.
		
		We say that $f$ is {\it essentially one-to-one} if there does not exist any disjoint measurable sets $A,B$ of positive $\mu$-measure such that
		$f(A)=f(B)$. The next result states that, within the world of bi-Lipschitz $\mu$-transformations, the property of being essentially one-to-one is a necessary condition for topological transitivity.
		
		\begin{lemma}\label{lex1}
		Let $(X,\mathcal B,\mu)$ be a finite measure space and $f:X\to X$ be a bi-Lipschitz $\mu$-transformation. If $T_f$ is topologically transitive, then $f$ is essentially one-to-one.  
		\end{lemma}
		\begin{proof} Suppose that $f$ is not essentially one-to-one, i.e., that there are disjoint sets $A$ and $B$ with $\mu(A)\cdot\mu(B) >0$ and $f(A) = f(B)$. We will show that $T_f$ is not topologically transitive by proving that $T_f(\mathcal{X})$ is actually not dense in $\mathcal{X}$. Remember that $\mathcal{X}\subset L^0(\mu)$ is any  Banach space of functions satisfying (H1)-(H4).
			
			As $\mu(A)<\infty$, $\rchi_{A}\in \mathcal{X}$ by (H1). Let $\epsilon>0$ be such that $c_2\epsilon<\frac12 \mu\big(f(B)\big)$, where $c_2>0$ is as in $(\ref{blc})$. Let $\delta>0$ be such that (H3) holds with the previously fixed  $\epsilon$. We will show that there is no $\varphi \in \mathcal{X}$ with $ \left\Vert \varphi\circ f - \rchi_A \right\Vert <\delta/2$. To obtain a contradiction assume that such a $\varphi$ exists and let $\psi=2\big(\varphi\circ f-\rchi_A\big)$. Let $B' =\left\{ x\in B:   \left|\varphi(f(x))\right| \ge \frac12\right\}.$ As $B'\subset B\subset X{\setminus} A$, we have that $\rchi_A=0$ on $B'$, thus $\left\vert\psi\right\vert\ge 1$ on $B'$. Moreover, by the choice of $\varphi$, we have that $\Vert \psi\Vert<\delta$. By (H3) applied to $S=B'$, we have that
			$\mu(B')<\epsilon$. By $(\ref{blc})$, 
	$\mu\left(f(B')\right)\le c_2\epsilon<  \frac{1}{2} \mu \left(f(B)\right).$ Let
			\begin{equation*}
			C= \left\{y \in f(B): |\varphi(y)| < \dfrac{1}{2}\right\}=f(B){\setminus} f(B'),\quad\textrm{thus}\quad \mu(C)> \frac12 \mu\left(f(B)\right).
			\end{equation*}
			Analogously, let 
			$$ A' =\left\{ x\in A:   |\varphi(f(x))| \le \frac12 \right \}\quad\textrm{and}\quad
			D= \left\{y \in f(A): |\varphi(y)|  > \dfrac{1}{2}\right\}=f(A){\setminus} f(A').
			$$
			Arguing exactly as above, we may ensure that $\mu(D)>\frac 12\mu\big(f(A)\big)$. Since $f(A) = f(B)$, we have that $C$ and $ D$ are two disjoint subsets of $f(B)$ both
			with measure greater than $\frac{1}{2} \mu (f(B))$, yielding a contradiction.
		\end{proof} 
		
		Lemma $\ref{lex1}$ leads to the following explicit example, which satisfies Conditions (C4) and (D4) but not the requirement that $f$ is one-to-one.
		
		\begin{example} Let $X=[0,1]$ be endowed with the Lebesgue measure $\mu$ and $f:X\to X$ be the bi-Lipschitz $\mu$-contraction defined by $x\mapsto \frac{x}{2}$ on $\left[0,\frac12\right)$ and $x\mapsto \frac{x}{2}-\frac14$ on $\left[\frac12,1\right]$, then $\limsup_{k\to\infty}{\mu\left( f^k(X)\right)}=0$
			but
			$T_f$ is not topologically transitive.
		\end{example}

\bibliographystyle{plain}
\bibliography{references.bib}

\end{document}